\documentclass[11pt]{amsart}

\usepackage{hyperref}
\usepackage{amsfonts,amssymb,amscd,amsmath,amsthm,enumerate,mathrsfs}
\newcommand{\abs}{\vspace{12pt}}

\DeclareMathOperator{\const}{const.}

\DeclareMathOperator{\diag}{diag}

\DeclareMathOperator{\Hess}{Hess}
\DeclareMathOperator{\card}{card}
\DeclareMathOperator{\Span}{span}

\DeclareMathOperator{\diam}{diam}

\DeclareMathOperator{\la}{\langle}
\DeclareMathOperator{\ra}{\rangle}

\newcommand{\R}{\mathbb{R}}
\newcommand{\Z}{\mathbb{Z}}

\newcommand{\N}{\mathbb{N}}

\newcommand{\MM}{\mathfrak{M}}

\newcommand{\U}{\mathcal{U}}

\renewcommand{\O}{\mathcal{O}}
\renewcommand{\P}{\mathcal{P}}

\newcommand{\M}{\mathcal{M}}

\newcommand{\e}{\varepsilon}

\newcommand{\gam}{\gamma}
\newcommand{\Gam}{\Gamma}
\newcommand{\al}{\alpha}

\newcommand{\lam}{\lambda}

\theoremstyle{plain}
\newtheorem{defn}{Definition}[section]
\newtheorem{lemma}[defn]{Lemma}

\newtheorem{thm}[defn]{Theorem}
\newtheorem{cor}[defn]{Corollary}

\theoremstyle{definition}
\newtheorem{remark}[defn]{Remark}

\newcommand{\mane}{Ma\~n\'e}
\newcommand{\Mane}{Ma\~n\'e }

\begin{document}

\hypersetup{pdftitle = {Uniqueness of shortest closed geodesics for generic Finsler metrics}, pdfauthor = {Jan Philipp Schr\"oder}}

\title[Shortest closed geodesics for generic Finsler metrics]{Uniqueness of shortest closed geodesics \\ for generic Finsler metrics}
\author[J. P. Schr\"oder]{Jan Philipp Schr\"oder}
\address{Faculty of Mathematics \\ Ruhr University \\ 44780 Bochum \\ Germany}
\email{\url{jan.schroeder-a57@rub.de}}

\begin{abstract}
In every conformal class of Finsler (or Riemannian) metrics on a closed manifold there exists a residual subset of Finsler metrics, such that, with respect to the residual Finsler metrics, in any non-trivial homotopy class of free loops there is precisely one shortest geodesic loop.
\end{abstract}

\maketitle


\section{Introduction and main result}

We consider Finsler metrics on closed manifolds, which can be thought of as a norm in each tangent space, while we do not assume reversibility, i.e. symmetry with respect to $v\mapsto -v$. Every Riemannian metric induces of course a Finsler metric, while it is also well-known that one can study the much more general Tonelli Lagrangian systems using Finsler metrics, cf. e.g. \cite{contreras}.

Let $M$ denote a closed manifold, i.e. compact, connected and $\partial M=\emptyset$. 

\begin{defn}\label{def finsler}
A function $F:TM\to\R$ is a \emph{Finsler metric} on $M$, if the following conditions are satisfied:
\begin{enumerate}
\item (smoothness) $F$ is $C^\infty$ at every $v\in TM, v\neq 0$,

\item (positive homogeneity) $F(a v) = a F(v)$ for all $v\in TM, ~ a \geq 0$,

\item (strict convexity) the fiberwise Hessian $\Hess(F^2|_{T_xM})$ of the square $F^2$ is positive definite at every vector $v\in T_xM-\{0\}$ for all $x\in M$.
\end{enumerate}
\end{defn}

We fix for the rest of the paper a Finsler metric $F$ on $M$.

In \cite{mane}, R. \Mane introduced a concept of genericity for Lagrangian systems, namely perturbing a given Lagrangian $L:TM\to\R$ by a potential function $V:M\to\R$ and considering the new Lagrangian $L' = L+V$. Properties of such generic Lagrangians have been studied in the literature, e.g. in \cite{bercont}. In the setting of Finsler metrics, we use a different notion of genericity, which is related to the before mentioned one via Maupertuis' principle. We write
\[ E := C^\infty(M)=C^\infty(M,\R), \quad E_+ := \{\lam \in E : \lam(x)>0 ~ \forall x \in M\} \]
and endow both sets with the $C^\infty$-topology.

\begin{defn}
We say that a property $\P$ of the Finsler metric $F$ is \emph{conformly generic}, if there exists a residual set $\O\subset E_+$ (i.e. $\O$ is the countable intersection of open and dense subsets of $E_+$), such that $\P$ holds for every Finsler metric of the form $\sqrt \lam\cdot F, ~ \lam\in \O$.
\end{defn}

\begin{remark}
\begin{itemize}
\item Note that the following two properties are left invariant under the perturbation $\sqrt \lam\cdot F$. (1) reversibility: $F(-v)=F(v)$ for all $v\in TM$. (2) Riemannian: $F(v)=\sqrt{g(v,v)}$ for a Riemannian metric $g$ on $M$.

\item Countable intersections of residual sets are again residual and, as $E$ is a Fr\'echet space, residual sets are dense in $E_+$.
\end{itemize}
\end{remark}

We denote by $\Gam(M)$ the set of all non-trivial homotopy classes of free, piecewise $C^1$ loops $c:\R/\Z\to M$. A classical way to obtain closed geodesics in $M$ is minimization of the $F$-length $l_F(c)=\int F(\dot c) dt$ in the homotopy classes $\gam\in \Gam(M)$. Namely, for all $\gam\in\Gam(M)$, the infimum $\inf_{c\in\gam} l_F(c)$ exists and is attained at closed geodesics in $\gam$ (cf. e.g. Theorem 8.7.1 (2) in \cite{bao}).

In the following, when we speak of $F$-geodesics $c:[a,b]\to M$, we assume that $c$ is parametrized such that $F(\dot c)=\const$. Also, if $c(t)$ is a geodesic and $t_0\in\R$, then we think of $c(t+t_0)$ as the same geodesic. We prove here the following theorem.

\begin{thm}\label{main thm}
Let $F$ be a Finsler metric on a closed manifold $M$, such that the set $\Gam(M)$ is countable. Then the following property is conformly generic: $F$ has in each non-trivial, free homotopy class $\gam\in\Gam(M)$ precisely one shortest closed geodesic.
\end{thm}

\abs

{\bf Structure of this paper.} We will approach the proof of Theorem \ref{main thm} through an application of a result due to R. \Mane from \cite{mane}. To keep the paper self-contained, we prove the special case used in this paper in Section \ref{appendix A}. In Section \ref{section free loops}, we study probability measures concentrated on loops of a given homotopy class and prove Theorem \ref{main thm}.

\section{An abstract result due to R. \Mane}\label{appendix A}

In this section, we prove an abstract result, following the ideas of R. \Mane (Section 3 in \cite{mane}). We work in the following setting, where all vector spaces are real vector spaces.
\begin{itemize}
\item $E$ is a Hausdorff topological vector space and has the following property: if $\phi:E\to \R$ is such that $\limsup \phi(f_n)\leq \phi(f)$, if $f_n\to f$, then the sublevels $\{\phi < c\}\subset E$ are open for all $c\in\R$,

\item $V$ is a topological vector space and $K\subset V$ is non-empty, compact, convex and metrizable,

\item $\varphi : E \times K \to \R$ is bilinear, sequentially continuous and seperates $K$ in the sense that
\[ \forall x,y\in K, x\neq y ~\exists f\in E: \quad \varphi(f,x)\neq \varphi(f,y). \]
\end{itemize}

We set
\begin{align*}
  m(f) & := \min \varphi(f, . )|_K, \\
 \M(f) & := \{ x\in K : \varphi(f,x) = m(f) \}.
\end{align*}

As a first observation, we have the following continuity property.

\begin{lemma}\label{M(f) semi-cont}
The map $f\in E \mapsto m(f)\in \R$ is sequentially continuous. Moreover, let $f_n \to f$ in $E$. Then for all sequences $x_n\in \M(f_n)$, any limit point $x$ of $\{x_n\}\subset K$ belongs to $\M(f)$.
\end{lemma}

\begin{proof}
If $x\in \M(f)$, then by the continuity of $\varphi$ 
\[ m(f) = \varphi(f,x) = \lim \varphi(f_n,x) \geq \limsup m(f_n). \]
Let $x_n\in\M(f_n)$, then after passing to a subsequence, assume that $x_n\to x$ in $K$. Then we have
\begin{align}\label{m cont}
 m(f_n) = \varphi(f_n,x_n) \to \varphi(f,x) \geq m(f),
\end{align}
showing that $\liminf m(f_n) \geq m(f)$. The first claim follows. Moreover, using $m(f_n)\to m(f)$, \eqref{m cont} shows that $x=\lim x_n\in \M(f)$.
\end{proof}

\begin{thm}[\mane]\label{thm mane}
There exists a residual subset $\O \subset E$ (i.e. a countable intersection of open and dense subsets), such that
\[ f\in \O \quad\implies\quad \card\M(f)=1. \]
\end{thm}

In the following, we fix a metric $d$ on $K$, whose induced topology coincides with the original topology on $K$, rendering $(K,d)$ a compact metric space. For a subset $M\subset (K,d)$ we write
\[ \diam M = \sup\{d(x,y) : x,y \in M \}. \]

\begin{proof}
Set $\O_n := \{ f\in E : \diam \M(f) < 1/n \}$ and $\O=\bigcap_{n\in\N} \O_n$. The density of $\O_n$ will follow from Lemma \ref{lemma 3.3} below. Suppose now that $f, f_k\in E$ with $f_k\to f$. By Lemma \ref{M(f) semi-cont}, if $x_k,y_k\in \M(f_k)$ are chosen such that $\limsup d(x_k,y_k) = \limsup \diam \M(f_k)$, we find limits $x,y\in \M(f)$ with $d(x,y)\geq \limsup \diam \M(f_k)$. This shows that $f\mapsto \diam\M(f)$ is sequentially upper semi-continuous and by our assumption on $E$, it follows that the sets $\O_n$ are open.
\end{proof}

Given a closed, convex subset $K_0\subset K$ and $f\in E$, we set
\[ \M_0(f) := \{ x \in K_0 : \varphi(f,x) = \min \varphi(f,.)|_{K_0} \}. \]

\begin{lemma}\label{lemma 3.2}
If $K_0\subset K$ is closed and convex, then
\[ \forall \e>0 ~ \exists f\in E ~ : \quad \diam \M_0(f)\leq \e. \]
\end{lemma}

\begin{proof}
For $x\neq y$ in $K_0$ let $f(x,y)\in E$ such that $\varphi(f(x,y),x-y)\neq 0$ and consider open neighborhoods $U(x,y)\subset K_0\times K_0$ of $(x,y)$, such that $\varphi(f(x,y),x'-y')\neq 0$ for all $(x',y')\in U(x,y)$. We find a sequence $U(x_n,y_n)$ covering $K_0\times K_0-\diag$.\footnote{$K_0\times K_0$ is a compact metric space. In general, let $(M,d)$ be a compact metric space and $D\subset M$ a closed subset, so $x\in D$ iff $d(x,D)=0$. Suppose that for all $x\in M-D$ we are given an open neighborhood $U_x\subset M-D$ of $x$ and for $n\in \N$ consider the open cover $\U_n := \{ U_x\}_{x\in M-D} \cup \{ B_d(x,1/n)\}_{x\in D}$ of $M$. Then $\U_n$ has a finite subcover and with $n\to\infty$, we obtain a countable subcover of $M-D$ of $\{U_x\}_{x\in M-D}$ by taking for each $n\in\N$ finitely many $U_x, x\in M-D$.} Setting $f_n := f(x_n,y_n)$, we find
\begin{align}\label{eq1}
\forall x,y\in K_0, x\neq y ~ \exists n\in \N : \quad \varphi(f_n,x-y)\neq 0.
\end{align}
For $n\in\N$ consider the linear map
\[ T_n : K_0 \to \R^n, \quad T_n(x) := (\varphi(f_1,x), ... , \varphi(f_n,x)). \]
Fix $\e>0$. We claim that
\begin{align}\label{eq2}
\exists n\in\N ~  \forall p \in \R^n: \quad \diam T_n^{-1}(p)\leq \e.
\end{align}

Proof of \eqref{eq2}: Suppose that for all $n\in \N$ we find $p_n\in \R^n$, such that $\diam T_n^{-1} (p_n) > \e$. Then we find a sequences $x_n,y_n\in K_0$, such that $T_n(x_n)=T_n(y_n)=p_n$ and $d(x_n,y_n)\geq \e$. Choosing convergent subsequences of $x_n,y_n$ by the compactness of $K_0$, we find limits $x,y\in K_0$ with $d(x,y)\geq \e$. By \eqref{eq1}, we find then $i\in\N$, such that $\varphi(f_i,x-y) \neq 0$ and by continuity of $\varphi$, we have for large $n\geq i$ (using the euclidean distance in $\R^n$)
\[ 0 = d(T_n(x_n),T_n(y_n)) \geq |\varphi(f_i,x_n-y_n)| \geq |\varphi(f_i,x-y)|/2 > 0, \]
a contradiction.

We fix $n$, such that \eqref{eq2} holds. Let $p\in T_n(K_0)$ be an exposed point, i.e. there exists $v\in (\R^n)^*$, such that $v|_{T_n(K_0)}$ attains its unique minimum at $p$ (such points exist by Straszewicz's theorem, since $T_n(K_0)\subset \R^n$ is convex and compact and hence has extremal points).
Writing $v = \sum_{i=1}^n v_i \cdot dx_i$ and using the linearity of $\varphi$, we obtain
\[ v\circ T_n = \varphi(f,.)|_{K_0} , \quad \text{where } f := \sum_{i=1}^n v_i \cdot f_i . \]
By definition of $v$ and $\M_0(f)$, we find
\[ \M_0(f) = T_n^{-1}(p). \]
The claim follows using \eqref{eq2}.
\end{proof}

\begin{lemma}\label{lemma 3.3}
If $f\in E$, then
\[ \forall ~ \e>0 , ~ U \subset E \text{ neighborhood of } f, ~ \exists f_*\in U : \quad \diam\M(f_*)\leq \e. \]
\end{lemma}

\begin{proof}
Fix $f\in E$ and $\e>0$. The set $K_0:=\M(f)\subset K$ is closed and convex, so by Lemma \ref{lemma 3.2}, we find $g\in E$, such that $m_0(g) := \min \varphi(g,.)|_{\M(f_0)}$ is attained in a set
\begin{align}\label{M_0 small}
 \M_0(g) \subset \M(f) ,\quad \diam \M_0(g)\leq \e/2.
\end{align}
We want to prove that
\begin{align}\label{proves lemma}
\lim_{t \to 0} ~ \diam \M(f+tg) \leq \e,
\end{align}
then the lemma follows. Observe that for $x\in \M_0(g)\subset \M(f)$ we have
\begin{align}\label{m(a) bound}
 m(f+tg) \leq \varphi(f+tg,x) = m(f)+t\cdot m_0(g).
\end{align}
For $x\in \M(f+tg)$ it follows, that
\begin{align*}
t\cdot \varphi(g,x) & = \underset{=m(f+tg)}{\underbrace{\varphi(f+tg,x)}} - \underset{\geq m(f)}{\underbrace{\varphi(f,x)}} \leq m(f+tg) - m(f) \\
& \stackrel{\eqref{m(a) bound}}{\leq} m(f)+t\cdot m_0(g) - m(f) = t\cdot m_0(g)
\end{align*}
and hence for $t>0$
\begin{align}\label{eq g}
\varphi(g,.)|_{\M(f+tg)} \leq m_0(g) .
\end{align}
Let now $x_t,y_t\in\M(f+tg)$, then by Lemma \ref{M(f) semi-cont}, any pair of limit points $x,y$, as $t\to 0$, belongs to $\M(f)$. Moreover, \eqref{eq g} shows that $\varphi(g,x), \varphi(g,y)\leq m_0(g)$, i.e. $x,y\in \M_0(g)$, hence $d(x,y)\leq \e/2$ by \eqref{M_0 small} and for $t>0$ small we have $d(x_t,y_t)\leq \e$. Hence, \eqref{proves lemma} follows.
\end{proof}

\section{Measures on free loops}\label{section free loops}

We return to the setting described in the introduction and begin with some notation. Let $|.|$ be the norm of a fixed Riemannian metric $g$ on the closed manifold $M$. For $b>0$ denote by $\M_b$ the set of probabilities supported in the compact set
\[ T_{\leq b}M := \{v\in TM :|v|\leq b \}. \]
Then $\M_b$ can be identified with a subset of the dual space $C^0(T_{\leq b}M)^*$. We endow $\M_b$ with the topology of weak*-convergence, i.e.
\[ \mu_n\to\mu \quad \iff \quad \int fd\mu_n\to \int fd\mu \quad \forall f \in C^0(T_{\leq b}M). \]
It is well-known that $\M_b$ is convex, metrizable and compact with respect to this topology.

\begin{defn}
Fix $\gam\in\Gam(M)$ and $b>0$. If $c:\R/\Z\to M$ is piecewise $C^1$ with $|\dot c|\leq b$, we define $\mu_c\in\M_b\subset C^0(T_{\leq b}M)^*$ by
\[ \int f d\mu_c =  \int_0^1 f(\dot c) dt \quad \forall f\in C^0(T_{\leq b}M). \]
and define $K_\gam^b\subset \M_b$ to be the closure of the convex hull of 
\begin{align*}
\big\{ \mu_c\in \M_b ~|~ c : \R/\Z \to M \text{ piecewise $C^1$, $c \in \gam$, $|\dot c|\leq b$} \big\} .
\end{align*}
\end{defn}

By the compactness and convexity of $\M_b$, the sets $K_\gam^b$ are always compact and convex. We will show that for large $b$, given the Finsler metric $F$, we always find particular minimizers of the action
\[ A_F : K_\gam^b \to \R, \quad A_F(\mu) := \int F^2 d\mu, \]
which are defined by shortest geodesic loops in $\gam$. Note that, since $F$ is continuous, $A_F$ is continuous by definition of the weak*-topology. Obviously, $A_F$ is also linear.

Using the compactness of $M$, we denote by $c_F\geq 1$ some constant, such that
\[ \frac{1}{c_F} \cdot F \leq |.| \leq c_F\cdot F \]
and writing $l_F, l_g$ for the lengths with respect to $F,g$, respectively, we define
\[ C_0(F,\gam) := c_F^2 \cdot \min_{c\in\gam} l_g(c) . \]

\begin{lemma}\label{min geod = min meas}
If $c:\R/\Z\to M$ is an $F$-geodesic loop, minimal with respect to $l_F$ in $\gam\in\Gam(M)$, then for all $b\geq C_0(F,\gam)$ the measure $\mu_c$ lies in $K_\gam^b$ and $\mu_c$ is minimal in $K_\gamma^b$ with respect to the action $A_F:K_\gam^b \to \R$.
\end{lemma}

\begin{proof}
Let $c_0:\R/\Z\to M$ be a $g$-geodesic loop, minimal in $\gam$ with respect to $l_g$. If $c:\R/\Z\to M$ is minimal in $\gam$ with respect to $l_F$, using $l_g(c_0)=|\dot c_0|, l_F(c)=F(\dot c)$, we find
\begin{align*}
 |\dot c| & \leq c_F\cdot F(\dot c) \leq c_F\cdot \max F(\dot c_0) \leq c_F^2\cdot |\dot c_0| = c_F^2\cdot \min_{c\in\gam} l_g(c) = C_0(F,\gam).
\end{align*}
We obtain $\mu_c\in K_\gam^b$ for $b\geq C_0(F,\gam)$. The $L^2$-Cauchy-Schwarz inequality shows for any piecewise $C^1$ curve $\al:[0,1]\to M$, that
\begin{align*}
 l_F(\al[0,1])^2 & = \la F(\dot \al), 1\ra_{L^2[0,1]}^2 \leq \|F(\dot \al)\|_{L^2[0,1]}^2 \cdot \|1\|_{L^2[0,1]}^2  = \int_0^1 F^2(\dot \al) dt \\
& = A_F(\mu_\al) 
\end{align*}
with equality if and only if the piecewise $C^0$-functions $1, F(\dot \al):[0,1]\to \R$ are linearly dependent. Hence
\begin{align}\label{CS}
l_F(\al[0,1])^2 \leq ~ A_F(\mu_\al) , \quad \text{ equality} \iff \exists k \in \R : F(\dot \al)=k \text{ a.e.}
\end{align}
and we find for any piecewise $C^1$ loop $c':\R/\Z\to M$ in $\gam$ with $|\dot c'|\leq b$, using $F(\dot c)=\const$ and the minimality of $c$ in $\gam$, that
\begin{align*}
 A_F(\mu_c) & = l_F(c )^2 \leq l_F(c')^2 \stackrel{\eqref{CS}}{\leq} A_F(\mu_{c'}) .
\end{align*}
Using the continuity and linearity of $A_F$, we find
\[ A_F(\mu_c) \leq A_F(\mu) \qquad \forall \mu\in K_\gam^b. \]
\end{proof}

In order to apply Theorem \ref{thm mane} to our situation, we introduce the relevant notation. Recall $E = C^\infty(M)$ and define a metric $d$ on $E$ by
\[ d(f,g) = \sum_{k \geq 0} \frac{1}{2^k}\frac{\|f-g\|_k}{1+\|f-g\|_k}, \]
where $\|.\|_k$ denotes the $C^k$-norm. This metric generates the $C^\infty$-topology on $E$. In particular, $E$ is a Hausdorff topological vector space and a set $U\subset E$ is open if for all $f\in U$ we find an open Ball $B_d(f,\e)\subset U$. Let $\phi:E\to \R$ be upper semi-continuous in the sense that $\limsup\phi(f_n) \leq \phi(f)$, if $f_n\to f$. We claim that the sets $\{f\in E : \phi(f)<c\}$ are open in $E$. Suppose the contrary, i.e. there exists $f\in E$ with $\phi(f)<c$ and for all $n\in\N$ there exists $f_n \in B_d(f,1/n)$, such that $\phi(f_n)\geq c$. But then $c\leq \limsup\phi(f_n) \leq \phi(f) < c$, contradiction.

Let
\[ \MM := \{ \mu \text{ positive, finite Borel measure on $M$}\}\subset C^0(M)^* , \]
which is metrizable, when endowed with the weak*-topology (cf. e.g. Theorem 8.3.2 in \cite{Bogachev}). Similar to the setting in \cite{bercont}, writing $\pi:TM\to M$ for the canonical projection, we define a projection
\begin{align*}
& \pi_*^F: \cup_{b>0}\M_b  \to \MM , \quad  \int f d(\pi_*^F\mu) := \int (f\circ \pi) \cdot F^2 d\mu \quad \forall f\in C^0(M).
\end{align*}
The map $\pi_*^F$ is linear and continuous and hence the set
\[ K := \pi_*^F(K_\gam^b) \subset V:= \Span(K) \subset C^0(M)^*  \]
is convex, compact and metrizable in $\MM$, suppressing $\gam,b$ in the notation for the moment. We define a bilinear, continuous map by integration:
\[ \varphi : E \times K \to \R , \quad \varphi(f ,\mu) := \int f d\mu . \]
Then $\varphi$ separates $K \subset C^0(M)^*$, since $E$ is dense in $C^0(M)$. Write
\begin{align*}
 \M_\gam^b(f) & := \{ \mu \in K  : \varphi(f,\mu) = \min \varphi(f, . )|_K \}.
\end{align*}
Theorem \ref{thm mane} can then be restated in our situation as follows.

\begin{cor}\label{prop 3.1}
In the setting described above, assuming $K_\gam^b\neq \emptyset$, there exists a residual set $\O_\gam^b\subset E$, such that
\[ f\in\O_\gam^b \quad \implies \quad \card \M_\gam^b(f) = 1. \]
\end{cor}

Using this result, we can prove our main theorem.

\begin{proof}[Proof of Theorem \ref{main thm}]
Note that for $b\geq C_0(F,\gam)=c_F^2\cdot \min_\gam l_g$ we have $K_\gam^b\neq\emptyset$ by Lemma \ref{min geod = min meas}. Recall $E_+ = \{\lam\in E: \lam(x)>0 ~ \forall x\in M\}$. For fixed $\gam\in\Gam(M)$, using the residual sets $\O_\gam^b$ from Corollary \ref{prop 3.1}, we consider the residual subset of $E_+$ defined by
\[ \O_\gam :=  \bigcap_{b\in \N, ~ b\geq C_0(F,\gam)} \O_\gam^b \cap E_+ . \]
Let $\lam\in \O_\gam$ and assume that there are two shortest $\sqrt\lam F$-geodesic loops $c_0,c_1:\R/\Z\to M$ in $\gam$. We choose
\[ b\in\N, \quad b\geq \max\{ C_0(\sqrt\lam F,\gam),C_0(F,\gam)\}, \]
then by definition of $\O_\gam$ we have $\lam\in \O_\gam^b$. By Lemma \ref{min geod = min meas}, the two measures $\mu_{c_0}, \mu_{c_1}$ evenly distributed on $\dot c_0,\dot c_1$ belong to $K_\gam^b$ and minimize the action $A_{\sqrt\lam F}:K_\gam^b\to \R$. Hence, if $\mu = \pi_*^F\hat \mu \in K$ for some $\hat \mu\in K_\gam^b$, we find
\begin{align*}
 \varphi(\lam, \pi_*^F\mu_{c_i}) & = \int \lam d(\pi_*^F\mu_{c_i}) = \int (\lam \circ \pi) \cdot F^2 d\mu_{c_i} = A_{\sqrt\lam F}(\mu_{c_i}) \\
& \leq A_{\sqrt\lam F}(\hat\mu) = \varphi(\lam,\pi_*^F\hat \mu) = \varphi(\lam, \mu) .
\end{align*}
This shows that  both $\pi_*^F\mu_{c_i}$ minimize $\varphi(\lam,.)|_K$ and, by definition of $\lam\in\O_\gam^b$, we have $\pi_*^F\mu_{c_0}=\pi_*^F\mu_{c_1}$. We obtain
\[ \int f d(\pi_*^F\mu_{c_0})= \int f d(\pi_*^F\mu_{c_1}) \quad \forall f\in C^0(M) , \]
showing that for all $f\in C^0(M)$ we have
\[ \int_0^1 (f\circ c_0) \cdot F(\dot c_0)^2 dt > 0 \iff  \int_0^1 (f\circ c_1) \cdot F(\dot c_1)^2 dt > 0 . \]
By considering bump functions $f:M\to [0,\infty)$ and using $F(\dot c_i)>0$ we obtain $c_0[0,1]=c_1[0,1]$. Hence, the $c_i$ are reparametrizations of each other. Finally, we set
\[ \O := \bigcap_{\gam\in\Gam(M)} \O_\gam , \]
which is residual in $E_+$, since $\Gam(M)$ is countable by assumption. The theorem follows.
\end{proof}

\end{document}